\documentclass{amsart}

\newtheorem{theorem}[subsection]{Theorem}
\newtheorem{lemma}[subsection]{Lemma}
\newtheorem*{claim}{Claim}

\theoremstyle{definition}
\newtheorem{definition}[subsection]{Definition}

\begin{document}

\title{A lower bound on minimal number of colors for links}

\author{Kazuhiro Ichihara}
\address{Department of Mathematics, 
College of Humanities and Sciences, Nihon University,
3-25-40 Sakurajosui, Setagaya-ku, Tokyo 156-8550, Japan}
\email{ichihara@math.chs.nihon-u.ac.jp}

\author{Eri Matsudo}
\address{Graduate School of Integrated Basic Sciences, Nihon University,
3-25-40 Sakurajosui, Setagaya-ku, Tokyo 156-8550, Japan}
\email{s6114M10@math.chs.nihon-u.ac.jp}

\dedicatory{Dedicated to Professor Yasutaka Nakanishi on the occasion of his 60th birthday}



\date{\today}

\maketitle

\begin{abstract}
We show that the minimal number of colors for all effective $n$-colorings of a link with non-zero determinant is at least $1+\log_2 n$.
\end{abstract}

\section{Introduction}

In \cite{Fox}, Fox introduced one of the most well-known invariants for knots and links, 
now called \textit{the Fox $n$-coloring}, or simply $n$-coloring 
for a natural number $n$.

In \cite{HararyKauffman}, 
Harary and Kauffman first studied 
\textit{the minimal number of colors} for colorings of knot or link diagrams. 
In \cite[Lemma 2.1]{Satoh}, Satoh showed that 
any non-trivial $n$-coloring for a knot diagram 
needs at least four colors if $n > 3$, and 
Kauffman and Lopez showed in \cite[Proposition 3.5]{KauffmanLopez} that 
the same holds for a non-splittable link diagram if $\gcd (n,3)=1$. 
Also, in \cite[Example 2.7]{Satoh}, it is stated that 
it can be shown similarly to \cite[Lemma 2.1]{Satoh} 
that any non-trivial $n$-coloring for a knot diagram 
needs at least five colors if $n > 7$. 
For the link case, the same was shown by Lopes and Matias 
in \cite[Theorem 1.4]{LopesMatias} 
if $n$ and the determinant of the link 
has the least common prime divisor greater than 7. 
Recently it is shown in \cite[Theorem 15]{GeJinKauffmanLopesZhang} 
the minimal number of colors of an $n$-colorable link 
with non-zero determinant is at least 6 if $n$ is a prime greater than 13. 

The results above for the knot case 
are extended by Nakamura, Nakanishi and Satoh 
in \cite[Theorem 1.1]{NakamuraNakanishiSatoh1} as follows. 
The minimal number of the distinct colors 
for all the non-trivially $n$-colored diagrams of a knot 
is greater than $ 1+ \log_2 n$ if $n$ is odd prime. 
They further showed in \cite[Theorem 2.7]{NakamuraNakanishiSatoh2} 
that it also holds when $n$ is odd and not necessarily prime 
for all effectively $n$-colored diagrams of a knot. 
See the next section for the definition of an effective coloring. 

It was pointed out in \cite[Section 1]{GeJinKauffmanLopesZhang} 
that their proof for \cite[Theorem 1.1]{NakamuraNakanishiSatoh1} 
can not be naturally extended to $n$-colorable links with non-zero determinant. 
In view of this, it is natural to ask what can we say for links. 
In this paper, we show the following for effectively $n$-colored diagrams of links.

\begin{theorem}\label{thm}
Let $n$ be a natural number. 
For any $n$-colorable link $L$ with non-zero determinant, 
let $C^*_n(L)$ be the minimal number of colors on 
effectively $n$-colored diagrams of $L$. 
Then $C^*_n (L) \geq 1+\log_2n$ holds. 
\end{theorem}

After preparing algebraic lemmas in Section 3, 
the proof of the theorem will be given in Section 4. 

\section{Effective colorings}

In this section, we set up our terminology. 

Let $L$ be a link and $D$ a regular diagram of $L$. 
We consider a map $\gamma:\{$arcs of $D\}\rightarrow \mathbb{Z}$. 
If, for a natural number $n$, 
$\gamma$ satisfies the condition 
$2\gamma(a)\equiv \gamma(b)+\gamma(c) \pmod n$ 
at each crossing of $D$ 
with the over arc $a$ and the under arcs $b$ and $c$, 
then $\gamma$ is called 
an \textit{$n$-coloring} on $D$. 
An $n$-coloring which assigns the same color to all the arcs of the diagram 
is called the \textit{trivial $n$-coloring}. 
Then $L$ is called an \textit{$n$-colorable} link 
if some diagram of $L$ admits a non-trivial $n$-coloring. 

Let us consider the cardinality of the image of $\gamma$ 
for a non-trivial $n$-coloring $\gamma$ on a diagram of $L$. 
We call the minimum of such cardinalities among 
all non-trivial $n$-colorings on diagrams of $L$ 
the \textit{minimal number of colors} of $L$ modulo $n$, 
and denote it by $C_n (L)$. 
We here use this notation following 
\cite{NakamuraNakanishiSatoh1, NakamuraNakanishiSatoh2}, 
but it is denoted by $mincol_n (L)$ in other papers 
including 
\cite{GeJinKauffmanLopesZhang, KauffmanLopez, LopesMatias}. 

In \cite[Lemma 1.6]{LopesMatias}, it was shown that 
if a non-splittable link $L$ with the determinant $\det L$ 
admits non-trivial $n$-colorings, 
then $C_n (L) = C_p (L)$ holds for a prime $p | (n , \det L)$. 
Thus we should consider more refined quantity than $C_n (L)$ 
for a non-prime $n$, for example, even $n$. 
In this paper, as in \cite{NakamuraNakanishiSatoh2}, 
we adapt the following definition, originally given in \cite{Kawauchi}. 

\begin{definition}
Suppose that a natural number $n$ 
has the prime decomposition $n=p_1^{e_1} \cdots p_s^{e_s}$ 
with primes $p_1, \cdots, p_s$. 
We say that 
a vector $\boldsymbol{x}=(x_1,\cdots,x_k)$ in $\mathbb{Z}^k$ 
is $p_i$-{\it trivial} 
if $x_1\equiv\cdots\equiv x_k \pmod {p_i}$ holds. 
If $\boldsymbol{x}$ is not $p_i$-trivial for any $p_i$, 
we say that $\boldsymbol{x}$ is \textit{$n$-effective}. 
Let $\gamma$ be an $n$-coloring on a diagram $D$ of a link $L$, 
and $\alpha_1,\cdots,\alpha_k$ the arcs of $D$. 
Put $x_i=\gamma(\alpha_i) \in \mathbb{Z}$ ($1 \le i \le k$). 
Then we define that $\gamma$ is an \textit{effective $n$-coloring} 
if $\gamma$ is non-trivial and $(x_1,\cdots,x_k)$ is $n$-effective.
\end{definition}

\section{Coloring matrix and the determinant of a link}

In this section, we recall some basic facts on 
$n$-colorings and determinants of links ($n \ge 2$), 
and give three algebraic lemmas used in the proof of Theorem \ref{thm} in the next section. 

First note that we have 
a system of homogeneous linear congruence equations modulo $n$ 
by regarding the arcs of a diagram $D$ of a link $L$ 
as algebraic variables and by setting up the equation at each crossing as: 
twice the over arc minus the sum of the under arcs equals zero modulo $n$. 
(See the next section for more details.) 
This system of congruence equations is called 
the \textit{coloring system of equations} for $D$. 
Then there is a natural correspondence between 
$n$-colorings of $D$ and solutions of the coloring system of equations for $D$. 

The coefficient matrix of the coloring system of equations for $D$ 
is called the \textit{coloring matrix} for $D$. 
It is known that 
the absolute value of the first minor of the coloring matrix for $D$ 
gives an invariant of the link $L$, 
which is coincident with the \textit{determinant} of $L$, denoted by $\det L$. 
See \cite{Lickorish} for example. 

Now we can see that there are non-trivial $n$-coloring of 
a link $L$ with $\det L \ne 0$ if and only if $n$ is not coprime to $\det L$. 
See \cite[Proposition 2.1]{LopesMatias} for example. 

\bigskip

In the following, we will give three algebraic lemmas 
used to proof our main theorem. 
To state them we prepare one more definition. 
Let $A$ be an integer matrix. 
If for a natural number $n$, 
there is an $n$-effective vector $\boldsymbol{x}$ 
such that $A\boldsymbol{x}\equiv \boldsymbol{0} \pmod n$, 
we say $\boldsymbol{x}$ is an \textit{$n$-effective solution} to $A$. 

The next two lemmas are the same as 
\cite[Lemma 2.1 and Lemma 2.2]{NakamuraNakanishiSatoh2} respectively. 
Thus we here omit the proofs. 

\begin{lemma}\label{lem1} 
Let $\boldsymbol{a}_1,\cdots,\boldsymbol{a}_k$ be vectors in $\mathbb{Z}^{k-1}$ 
such that $\boldsymbol{a}_1+\cdots+\boldsymbol{a}_k=\boldsymbol{0}$ with $k \ge 3$. 
Let $A=(\boldsymbol{a}_1,\cdots,\boldsymbol{a}_k)$ be 
the $(k-1)\times k$ matrix. 
Then the following are equivalent for $n \ge 2$.
\begin{enumerate}
\item
$A\boldsymbol{x}\equiv \boldsymbol{0} \pmod n$ 
has an $n$-effective solution $\boldsymbol{x}_0 = {}^t (x_1, \cdots, x_k)$. 
\item
$A'\boldsymbol{x}\equiv \boldsymbol{0} \pmod n$ 
has an $n$-effective solution $\boldsymbol{x}'_0 = {}^t (x'_1, \cdots, x'_{k-1}, 0)$ 
for $A'=(\boldsymbol{a}_1,\cdots,\boldsymbol{a}_{k-1},\boldsymbol{0})$.
\end{enumerate}
\end{lemma} 


\begin{lemma}\label{lem2} 
Let $\boldsymbol{a}_1,\cdots,\boldsymbol{a}_{k-1}$ be vectors in $\mathbb{Z}^{k-1}$ 
with $k \ge 3$. 
Let $A=(\boldsymbol{a}_1,\cdots,\boldsymbol{a}_{k-1},\boldsymbol{0})$ 
be the $(k-1)\times k$ matrix. 
Then the following are equivalent.
\begin{enumerate}
\item
$A \boldsymbol{x} \equiv \boldsymbol{0} \pmod n$ has 
an $n$-effective solution $\boldsymbol{x}_0 = {}^t (x_1, \cdots, x_{k-1}, 0)$.
\item
$A' \boldsymbol{x} \equiv \boldsymbol{0} \pmod n$ has 
an $n$-effective solution $\boldsymbol{x}'_0= {}^t (x'_1, \cdots, x'_{k-1}, 0)$, 
where $A'$ is one of the following;
\begin{enumerate}
\item
$A'=(\boldsymbol{a}_1,\cdots,\boldsymbol{a}_{i-1},\boldsymbol{a}_i+\lambda\boldsymbol{a}_j,\boldsymbol{a}_{i+1},\cdots,\boldsymbol{a}_{k-1},\boldsymbol{0})$ with $\lambda \in \mathbb{Z}$, 
\item
$A'=(\boldsymbol{a}_1,\cdots,\boldsymbol{a}_{i-1},\boldsymbol{a}_j,\boldsymbol{a}_{i+1},\cdots,\boldsymbol{a}_{j-1},\boldsymbol{a}_i,\boldsymbol{a}_{j+1},\cdots,\boldsymbol{a}_{k-1},\boldsymbol{0})$, or 
\item
$A'=(\boldsymbol{a}_1,\cdots,\boldsymbol{a}_{i-1},-\boldsymbol{a}_i,\boldsymbol{a}_{i+1},\cdots,\boldsymbol{a}_{k-1},\boldsymbol{0})$. 
\end{enumerate}
\end{enumerate}
\end{lemma}


The next lemma is implicitly used 
in \cite[Proposition 2.3]{NakamuraNakanishiSatoh2}. 
We here include a brief proof for completeness. 

\begin{lemma}\label{lem3}
Any matrix $A$ with $m$ rows 
can be transformed to the matrix $A'$ below 
by fundamental transformations 
without multiplying an integer other than $\pm 1$ 
to a row or a column of $A$.
\[
  A' = \left(
    \begin{array}{ccccccc}
      d_1 & 0 & \cdots & 0 & 0 & \cdots & 0 \\
      0 & d_2 & \ddots & \vdots& \vdots & \cdots & \vdots \\
      \vdots & \ddots & \ddots & 0 & \vdots & \cdots & \vdots \\
      0 & \cdots & 0 & d_m & 0 & \cdots & 0
    \end{array}
  \right)
\]
Here $d_i$'s are all natural numbers and $d_i$ is divisible by $d_{i-1}$.
\end{lemma}

\begin{proof}
Let $A=(a_{ij})$ be a $t\times m$ matrix. 
The transformations which we use here are as follows.
\begin{enumerate}
\item
Adding the $i$th column multiplied by an integer $k$ to the $j$th column.
\item
Multiplying $-1$ to the $i$th column.
\item
Exchanging the $i$th column for the $j$th column.
\end{enumerate}
We transform $A$ to a matrix $A_1$ 
such that the $(1,1)$-entry is positive and 
is smaller than or equal to 
the absolute values of all the entries of all the matrices obtained from $A$ 
by performing (1), (2), (3) repeatedly. 
We here note that 
all the entries of $A'$ are divisible by $a_{11}$, 
for, the otherwise we could find another matrix 
obtained by using (3) 
with an entry smaller than $a_{11}$, 
contradicting the assumption on $A_1$. 
Thus by using (3) repeatedly, $A_1$ is transformed to 
%
%
\[
  A_2 = \left(
    \begin{array}{ccccc}
      a_{11} & 0 & \cdots & \cdots & 0 \\
      0 & a_{22} & a_{32} & \cdots& a_{t2} \\
      \vdots & a_{23} & \cdots & \cdots & a_{t3} \\
      \vdots & \vdots & \ddots & \ddots & \vdots \\
      0 & a_{2m} & \cdots & \cdots & a_{tm}
    \end{array}
  \right)
\]
We apply this procedure for the submatrix of $A_2$ 
obtained by deleting the first column and row, 
and continue the same repeatedly. 
Finally we obtain the desired matrix 
\[
  A' = \left(
    \begin{array}{ccccccc}
      d_1 & 0 & \cdots & 0 & 0 & \cdots & 0 \\
      0 & d_2 & \ddots & \vdots& \vdots & \cdots & \vdots \\
      \vdots & \ddots & \ddots & 0 & \vdots & \cdots & \vdots \\
      0 & \cdots & 0 & d_m & 0 & \cdots & 0
    \end{array}
  \right)
\]
such that $d_1$ is at least $1$ and $d_i$ is divisible by $d_{i-1}$ ($1 \le i \le m$).
\end{proof}

\section{Proof of theorem}

\begin{proof}[Proof of Theorem \ref{thm}]
Let $D$ be a diagram of $L$ with $k$ crossings 
admitting an $n$-effective $n$-coloring $\gamma$. 
Let $\alpha_1,\cdots,\alpha_k$ be the arcs of $D$, 
and $q_1,\cdots,q_k$ be the crossings of $D$. 
Let $x_1 = \gamma(\alpha_1),\cdots, x_k=\gamma(\alpha_k)$ 
be colors (integers) on $\alpha_1,\cdots,\alpha_k$, and 
$l$ be the number of the distinct colors on $D$. 
Then we will show that $l \ge 1 + \log_2 n$ holds.

We recall the construction of the coloring matrix of $\gamma$. 
Precisely, associated to $\gamma$, we obtain the coloring matrix 
$A=(a_{ij})$ with $a_{ij}\in\mathbb{Z}$, 
which is a $k \times k$ matrix, as follows. 
$$
a_{ij}=
\begin{cases}
-2 & \text{ if $\alpha_j$ is the over arc on $q_i$,} \\
1 & \text{ if $\alpha_j$ is the under arc on $q_i$, and}\\
0 & \text{otherwise.}
\end{cases}
$$
We put  $\boldsymbol{x}_0={}^t (x_1,\cdots,x_k) \in \mathbb{Z}^k$. 
Since the coloring $\gamma$ is an effective $n$-coloring, 
the vector $\boldsymbol{x}_0$ gives an $n$-effective solution 
of the congruence equations $A\boldsymbol{x}\equiv \boldsymbol{0} \pmod n$. 

The next procedure is the key of our proof. 
Let us produce another matrix $A_1$ from $A$ as follows: 
First, take $x_k$. 
If $x_k=x_{k-1}$, 
add the $k$th column to the $(k-1)$th column and delete the $k$th column, 
and go to the next step. 
If $x_k\neq x_{k-1}$ and $x_k= x_{k-2}$, 
add the $k$th column to the $(k-2)$th column and delete the $k$th column, 
and go to the next step. 
Repeat this procedure in turn until the $1$st column. 
If $x_k\neq k_{k-1},\cdots,x_k\neq x_1$, go to the next step. 
Next, take $x_{k-1}$ and repeat the same procedure. 
We perform this until we take $x_2$ and $x_2\neq x_1$. 
Let $A_1$ be the matrix so obtained. 
We here remark that the number of columns of $A_1$ is equal to $l$. 
That is, $A_1$ is a $k \times l$-matrix. 

Let $\boldsymbol{y}_0={}^t (y_1,\cdots,y_l)$ be 
the vector obtained from $\boldsymbol{x}_0$ 
by deleting the entries 
which correspond to the columns deleted in the procedure above. 
Then, by the procedure above, 
this $\boldsymbol{y}_0$ gives 
a solution to $A_1\boldsymbol{y}\equiv\boldsymbol{0} \pmod n$. 
Moreover $\boldsymbol{y}_0$ is $n$-effective, 
since the set of entries of $\boldsymbol{y}_0$ is 
just equals to the set of entries of $\boldsymbol{x}_0$.

Since the entries on each row of $A$ are $-2$, 1, 1, and 0's, 
we obtain a column with only 0's by adding all the other columns to a fixed column. 
Note that the same holds for $A_1$. 
Because, by the procedure making $A_1$ from $A$, 
the sum of the column vectors of $A_1$ is equal to that for $A$. 
Then, using fundamental transformations, 
we can deform $A_1$ to the $k\times l$ matrix 
$A_2=(\mathbf{a}_1,\cdots,\mathbf{a}_{l-1},\boldsymbol{0})$ 
by adding all the other column to the $l$th column. 
Note that all the columns other than the $l$th are 
shared by $A_1$ and $A_2$. 
Then, by Lemma \ref{lem1} ((1) $\to$ (2)), 
$A_2\boldsymbol{y}\equiv\boldsymbol{0} \pmod n$ has 
an $n$-effective solution $\boldsymbol{y}_1$. 

Here we can see that $\mathrm{rank} A_2 = l-1$ as follows. 
Since $\det (L)$ is equal to 
the absolute value of a first minor of $A$, 
which is not equal to 0 by the assumption, 
any $k-1$ vectors among the $k$ column vectors of $A$ 
are linearly independent. 
That is, $\mathrm{rank} A = k-1$. 
Then, since the procedure making $A_1$ from $A$ 
is comprised of fundamental transformations of columns 
and deleting column vectors, 
together with the property that the sum of the column vectors of $A_1$ 
is $\boldsymbol{0}$, 
the number of linearly independent column vectors of $A_1$ 
is equal to $l-1$, 
implying that $\mathrm{rank} A_2 = l-1$. 

It follows that there are $l-1$ vectors 
which are linearly independent 
among the $k$ row vectors of $A_2$. 
By deleting the other row vectors from $A_2$, 
we obtain the matrix $A_3$, which is a $(l-1) \times l$ matrix. 
Note that the sum of the column vectors of $A_2$ is $\boldsymbol{0}$ still. 

Since the set of row vectors of $A_3$ is just a subset of that of $A_2$, 
the vector $\boldsymbol{y}_1$ also gives 
an $n$-effective solution to $A_3\boldsymbol{y}\equiv\boldsymbol{0} \pmod n$. 

Here, let $B$ be the matrix obtained from $A_3$ by deleting the $l$th column. 
Note that $\det B \ne 0$ since $\mathrm{rank} B = \mathrm{rank} A_3 = l-1$.

We deform $A_3$ to the next $A'_3$ by applying Lemma \ref{lem3} 
to $B$ as a part of $A_3$. 
\[
  A'_3 = \left(
    \begin{array}{ccccc}
      d_1 & 0 & \cdots & 0 & 0 \\
      0 & d_2 & \ddots & \vdots& \vdots \\
      \vdots & \ddots & \ddots & 0 & \vdots \\
      0 & \cdots & 0 & d_m & 0
    \end{array}
  \right)
\]
Here we can have $d_1 \ge 1$ and $d_i$ is divisible by $d_{i-1}$. 
By Lemma \ref{lem2} ((1) $\to$ (2)), 
together with the fact that 
if admitting an $n$-effective solution is invariant 
under the fundamental transformations of rows, 
$A'_3\boldsymbol{y}\equiv\boldsymbol{0} \pmod n$ 
has an $n$-effective solution $\boldsymbol{y}'_1 ={}^t (y_1,\cdots,y_{l-1},0)$. 

Now we follow the argument developed 
in \cite[Proposition 2.3 (ii)]{NakamuraNakanishiSatoh2}. 
Since $\boldsymbol{y}'_1$ is an $n$-effective solution, 
for each prime factor $p_i$ of $n$, 
there is $y_i$ which is not congruence to $0 \pmod {p_i}$. 
For this $y_i$, we have $d_i \equiv 0 \pmod {p_i}$ from $d_i y_i \equiv 0 \pmod n$. 
Thus $\det B = d_1 \cdots d_m \equiv 0 \pmod {p_i}$ for any $p_i$. 
This implies that $\det B \equiv 0 \pmod n$. 
Since $\det B \ne 0$, it concludes that $|\det B| \ge n$. 

\medskip

On the other hand, in the following, we show that $|\det B| \le 2^{l-1}$. 

By the definition of $A$, 
each row of $A$ includes $\{1,1,-2\}$ with the other entries are all $0$. 
By the procedure making $A_1$ from $A$, 
each row of $A_1$ includes either 
(i) $\{1,1,-2\}$, (ii) $\{2,-2\}$, (iii) $\{1,-1\}$ 
with the other entries are all $0$. 
Since the entries in a row vector of $A_2$ 
is just those for $A_1$ with the $l$th entry deleted, 
each row of $A_2$ includes either above (i), (ii), (iii), or 
(iv) $\{1,1\}$, (v) $\{1,-2\}$, (vi) $\{1\}$, (vii) $\{-1\}$ , (viii) $\{2\}$ , (ix) $\{-2\}$ 
with the other entries are all $0$. 
Each row vector of $A_3$ is either of type (i) to (ix), 
since $A_3$ is obtained from $A_2$ by just deleting a number of rows. 
Further, in the same way as above, 
we see that each row vector of $B$ is either of type (i) to (ix) also. 

Now, to complete the proof of Theorem \ref{thm}, it suffices to show the next claim. 

\begin{claim}
Any square matrix $M$ of size $\mu$ with row vectors 
each of which is either of type \textup{(i)} to \textup{(ix)} 
has the determinant $\det M$ with $| \det M| \le 2^\mu$. 
\end{claim}

\begin{proof}
We show this by induction of the size $\mu$. 

If $\mu=1$, then $M$ is either $(1)$, $(-1)$, $(2)$, $(-2)$, 
and so, we have $|\det M| \le 2$. 

Assume that $|\det M| \le 2^\mu$ holds for $\mu \le \nu -1$, 
and consider the case $\mu = \nu$. 

If some of the row vector of $M$ is either of type (iii), (iv), (vi), (vii), (viii), (ix), 
then by using the cofactor expansion along the row, 
we have $|\det M| \le 2^\mu$ by the assumption as the desired. 

If some of the row vector of $M$ is of type (ii), 
then by the fundamental transformation, 
$M$ is deformed into another matrix $M'$ having a row, which is of type (viii). 
This $M'$ may not satisfy the assumption of the induction 
at the column corresponding to that of $M$ including $2$ in the row. 
For example, $M'$ may contain a row including $\{ 1, 2, -2\}$. 
However, applying the cofactor expansion to $M'$ along the row of type (viii), 
$|\det M'|$ is calculated as 2 times 
the absolute value of the determinant of 
the minor matrix which satisfies the assumption of the induction. 
Thus we have $|\det M| \le 2^\mu$ as the desired. 

If some of the row vector of $M$ is of type (v), 
then by the fundamental transformation, 
$M$ is deformed into another matrix having a row, which is of type (iii). 
Applying the same argument as above, 
we also have $|\det M| \le 2^\mu$ as the desired. 

The remaining case is just for all the row vectors of $M$ are of type (i). 
In this case, the sum of all the column vectors of $M$ must be $\boldsymbol{0}$. 
This means that $\det M =0$ in this case, obviously satisfying $|\det M| \le 2^\mu$. 
\end{proof}

Consequently, we have $n \le |\det B| \le 2^{l-1}$, that is, $n \le 2^{l-1}$, 
equivalently, $1+ \log_2 n \le l$. 
This completes the proof. 

\end{proof}

%
%

\section*{Acknowledgement}
The authors would like to thank Yasutaka Nakanishi, 
Shin Satoh and Jun Ge for useful discussions in this topic. 
The first author is partially supported by JSPS KAKENHI Grant Number 26400100.

\end{document}